\documentclass[reqno]{amsart}
\usepackage[utf8]{inputenc}
\usepackage{stackengine}
\usepackage{amssymb}
\usepackage{amsthm}
\usepackage{amsmath}
\usepackage{enumerate}
\usepackage{mathtools}
\usepackage{comment}
\usepackage{todonotes}
\usepackage{xcolor}\mathtoolsset{showonlyrefs}
\usepackage{hyperref}
\usepackage{dsfont}

\newtheorem{thm}{Theorem}[section]
 
 \newtheorem{lem}[thm]{Lemma}
 \newtheorem{prop}[thm]{Proposition}

 \newtheorem{defn}[thm]{Definition}
  \newtheorem{ex}[thm]{Example}
 \newtheorem{rem}[thm]{Remark}

 \def\Xint#1{\mathchoice
   {\XXint\displaystyle\textstyle{#1}}%
   {\XXint\textstyle\scriptstyle{#1}}%
   {\XXint\scriptstyle\scriptscriptstyle{#1}}%
   {\XXint\scriptscriptstyle\scriptscriptstyle{#1}}%
   \!\int}
\def\XXint#1#2#3{{\setbox0=\hbox{$#1{#2#3}{\int}$}
     \vcenter{\hbox{$#2#3$}}\kern-.5\wd0}}

\def\dashint{\Xint-}

\def\upint{\mathchoice%
    {\mkern13mu\overline{\vphantom{\intop}\mkern7mu}\mkern-20mu}%
    {\mkern7mu\overline{\vphantom{\intop}\mkern7mu}\mkern-14mu}%
    {\mkern7mu\overline{\vphantom{\intop}\mkern7mu}\mkern-14mu}%
    {\mkern7mu\overline{\vphantom{\intop}\mkern7mu}\mkern-14mu}%
  \int}

\def\wlim{\mathrel{\ensurestackMath{\stackon[1pt]{\rightharpoonup}{\scriptstyle\ast}}}}
\title{Weak differentiability of metric space valued Sobolev maps}
\author{Paul Creutz}
	\address{
Max Planck Institute for Mathematics,
Vivatsgasse 7,
53111 Bonn,
Germany}
\email{Paul.Creutz@ish.de}

\author{Nikita Evseev}
\address{
 Institute for Advanced Study in Mathematics, Harbin Institute of Technology, 150006 Harbin, China, and
 Sobolev Institute of Mathematics, 4 Academic Koptyug avenue,  630090 Novosibirsk, Russia.
	}
\email{evseev@math.nsc.ru}

\begin{document}

\maketitle
\begin{abstract}
We show that Sobolev maps with values in a dual Banach space can be characterized in terms of weak derivatives in a weak* sense. Since every metric space embeds isometrically into a dual Banach space, this implies %allows for 
a characterization of metric space valued Sobolev maps in terms of such derivatives. %This generalizes previously known results for the case of separable target spaces. 
Furthermore, we investigate for which target spaces Sobolev maps are weak* differentiable almost everywhere. %It turns out that this is equivalent to the dual Banach space satisfying a weak* version of the Radon--Nikodym property.
\end{abstract}

\section{Introduction}
\subsection{Main results}
We study first-order Sobolev maps $u \colon \Omega \to X$ where $\Omega \subset \mathbb{R}^n$ is a bounded domain and $X$ a complete metric space. Several equivalent definitions of the Sobolev spaces $W^{1,p}(\Omega;X)$ have been established over the last 30 years. In particular nowadays it is well agreed upon which maps $u \colon \Omega \to X$ lie in $W^{1,p}(\Omega;X)$ and which do not, see \cite[Chapter 10]{HKST2015}, \cite{Reshetnyak2004} and \cite{AGS:13} for an overview. The classical approach for real valued functions is to define $W^{1,p}(\Omega)$ as the collection of those $f\in L^p(\Omega)$ which for every $j=1, \dots , n$ have a weak partial derivative $f^j\in L^p(\Omega)$ satisfying
\begin{equation}
\label{eq:classical_weak_derivative}
\int_\Omega\frac{\partial \varphi}{\partial x_j}(x)\cdot f(x)\ \textnormal{d}x 
= -\int_\Omega\varphi(x)\cdot f^j(x) \ \textnormal{d}x \quad \text{ for every } \varphi \in C^\infty_0(\Omega).
\end{equation}
The aim of this article is to provide a characterization of similar flavour for $X$-valued maps which is compatible with the well-established metric definitions of $W^{1,p}(\Omega ;X)$.

One can define $L^p(\Omega;X)$ as the collection of those measurable and essentially separably valued $f\colon \Omega \to X$ such that $d(f(-),z_0) \in L^p(\Omega)$ for every $z_0\in X$. However in general $X$ does not carry a linear structure and hence integrals as in~\eqref{eq:classical_weak_derivative} are not defined for $X$-valued maps. Nevertheless an auxiliary linear structure can be introduced by embedding $X$ isometrically into a Banach space $V$. For example the Kuratowski map $\kappa \colon X\to \ell^{\infty}(X)$  provides an isometric embedding into the Banach space $\ell^\infty(X)$ of bounded functions on $X$. Thus a natural approach is to define first Banach-valued Sobolev maps, and then metric valued ones as those which postcomposed with an isometric embedding into a Banach space are Sobolev. 

Functions in $L^p(\Omega;V)$ are Bochner integrable and hence it is tempting to define $W^{1,p}_{\textnormal{Ban}}(\Omega;V)$ as the collection of those $f\in L^p(\Omega; V)$ which have weak partial derivatives $f^j\in L^p(\Omega;V)$ that satisfy \eqref{eq:classical_weak_derivative} in the sense of Bochner integrals. For Banach valued maps such definition goes back to Sobolev himself \cite{Sob:63} and was studied classically by the PDE and functional analysis communities, see e.g.\ \cite{Sch87, CH1998, SS05,  HNVW16, AK18}. %\todo{P: If you have ideas for improvements of this list feel free to add or delete things. Only I would certainly cite the last since it is recent and seems somewhat related.} 
More recently it was introduced for metric valued maps in \cite{HT:08} and then used in a bunch of papers such as \cite{WZ:09, Haj:09, BMT:13, HST:14}.  However, as observed in \cite{CJPA2020} and \cite{Evs21}, such definition of $W^{1,p}_{\textnormal{Ban}}(\Omega;V)$  is consistent with the metric definitions of $W^{1,p}(\Omega;V)$ only if $V$ has the Radon--Nikodym property. This is problematic since many spaces of geometric interest %spaces such as the Heisenberg group or even $S^1$ endowed with the angular metric cannot
do not isometrically embed into a Banach space with this property, see \cite[Theorem 1.6]{CK:06} and \cite[Remark~4.2]{Cre:22}. Indeed the metric Sobolev space defined in terms of $W^{1,p}_{\textnormal{Ban}}(\Omega;\ell^{\infty}(X))$ and the Kuratowski embedding $\kappa \colon X\to \ell^\infty(X)$ is always trivial \cite{CE:2021}. To fix this issue we propose the following definition.
\begin{defn}
\label{def:w*}
    Let $\Omega \subset \mathbb{R}^n$ be a bounded domain, $p\in[1,\infty)$ and $V^*$ be a dual Banach space. The space $ L^{p}_*(\Omega; V^*)$ consists of those functions $f\colon \Omega\to V^*$ that are weak* measurable and for which the upper integral $\upint_\Omega \|f(x)\|^p \ \textnormal{d}x$ is finite.\\
A function $f$ lies in the Sobolev space $W^{1,p}_*(\Omega ; V^*)$ if $f\in L^p(\Omega ;V^*)$ and for every $j=1, \dots , n$ there is a function $f^j\in L^p_*(\Omega;V^*)$ such that
\begin{equation}
\label{eq:def-w*}
\int_\Omega\frac{\partial \varphi}{\partial x_j}(x)\cdot f(x)\ \textnormal{d}x 
= -\int_\Omega\varphi(x)\cdot f^j(x) \ \textnormal{d}x \quad \text{ for every } \varphi \in C^\infty_0(\Omega)
\end{equation}
in the sense of Gelfand integrals.
\end{defn}
Definition~\ref{def:w*} is a slight variation on Definition~1.3 in \cite{CE:2021} where additionally it is assumed that $x \mapsto \|f(x)\|$ is measurable for functions $f\in L^p_*(\Omega;V^*)$, see the discussion in Section~\ref{sec:meas-norm}. %If $V$ is separable this is implied automatically by Definition~\ref{def:w*}. 
The main result of this article is the equivalence of Definition~\ref{def:w*} with the metric characterizations of $W^{1,p}(\Omega; V^*)$.
\begin{thm}
\label{thm:main}
    Let $\Omega \subset \mathbb{R}^n$ be a bounded domain, $p\in (1,\infty)$ and $V^*$ be a dual Banach space. Then $W^{1,p}_*(\Omega, V^*)=W^{1,p}(\Omega; V^*)$.
\end{thm}
Indeed $W^{1,p}_*(\Omega, V^*)$ and $W^{1,p}(\Omega; V^*)$ are not only equal as sets but also that the corresponding Sobolev norms are equivalent, see Section~\ref{sec:uniq}. %Since $\ell^\infty(X)$ is dual to $\ell^{1}(X)$, 
Concerning the metric Sobolev space $W^{1,p}_*(\Omega; X)$ one can now choose an isometric embedding $\iota \colon X\to V^*$ into a dual Banach space $V^*$ (e.g.\ $\ell^\infty(X)$), define
\begin{equation}
\label{eq:def-w*-met}
    W^{1,p}_*(\Omega; X):=\left\{u \colon \Omega \to X \ | \ \iota \circ u \in W^{1,p}_*\left(\Omega; V^*\right) \right\}.
\end{equation}
and deduce from Theorem~\ref{thm:main} that $W^{1,p}_*(\Omega; X)=W^{1,p}(\Omega; X)$.

In \cite{CJP22} and \cite{CE:2021} the equality of $W^{1,p}_*(\Omega, V^*)$ and $W^{1,p}(\Omega; V^*)$ was shown if the predual $V$ is separable. Indeed every separable metric space $X$ embeds isometrically into the dual-to-separable Banach space $\ell^{\infty}(\mathbb{N})$, and for such $X$ a consistent definition of $W^{1,p}_*(\Omega; X)$ in terms of $W^{1,p}_*(\Omega; \ell^{\infty}(\mathbb{N}))$ was already proposed in \cite{CE:2021}. This approach however had the aesthetic drawback that first one had to define $W^{1,p}_*(\Omega; V^*)$ for dual-to-separable Banach spaces (which are usually not separable themself) as to then obtain a definition of $W^{1,p}_*(\Omega; X)$ which only worked for separable $X$.

The inclusion $W^{1,p}_*(\Omega; V^*)\subset W^{1,p}(\Omega; V^*)$ was already shown for arbitrary dual spaces in \cite{CE:2021}. The proof of the other inclusion for dual-to-separable ones in \cite{CJP22, CE:2021} however relies on the almost everywhere weak* differentiability of absolutely continuous curves. In particular the proof of the equality therein does not apply for arbitrary dual spaces since e.g.\ the Kuratowski embedding $\kappa\colon [0,1] \to \ell^{\infty}\left([0,1]\right)$ is nowhere pointwise weak* differentiable, see Example~\ref{ex:not-diff} below. In Section~\ref{sec:pointwise} we will introduce a class of dual spaces which generalizes dual-to-separable ones as well spaces with the Radon-Nikodým-property, and characterizes the almost everywhere weak* differentiability of absolutely continuous curves, see Definition~\ref{def:leb-dens} below. 
\subsection{Organization}
First in Section~\ref{sec:prel} we set up notation and recall some standard results concerning measurability, integrals and metric valued Sobolev maps. The main part is then Section~\ref{sec:weak-weak*}. First in Subsections~\ref{sec:hard_inc} and~\ref{sec:triv_inc} we prove Theorem~\ref{thm:main}. Then in Subsection~\ref{sec:uniq} we discuss the nonuniqueness of weak weak* derivatives and how to define a suitable Sobolev norm on $W^{1,p}_*(\Omega ;V^*)$. In Subsection~\ref{sec:meas-norm} we discuss the mentioned problem concerning the measurability of $x\mapsto \|f^j(x)\|$ for weak weak* derivatives $f^j$. Finally in Subsection~\ref{sec:pointwise} we investigate the pointwise almost everywhere weak* differentability and introduce a condition that guarantees it for absolutely continuous curves.
\subsection{Acknowledgements}
The first-named author wants to thank Matthew Romney and Elefterios Soultanis for helpful discussions.
In turn, the second-named author is grateful to Sergey Basalaev and Konstantin Storozhuk for fruitful comments.
\section{Preliminaries}
\label{sec:prel}
Throughout this section let $\Omega \subset\mathbb{R}^n$ be a bounded domain, $p\in [1,\infty]$, $V$ be a Banach space and $X$ be a complete metric space.
\subsection{Measurability and integrals}
\label{subsec:meas-int}
We call a function $f\colon \Omega \to X$ \emph{measurable} if it is measurable with respect to the Borel $\sigma$-algebra on $X$ and the $\sigma$-algebra of Lebesgue measurable sets on $\Omega$. For $p\in [1,\infty)$ (resp.\ $p=\infty$) we denote by $L^p(\Omega)$ the space of measurable functions $f\colon \Omega \to \mathbb{R}$ for which the Lebesgue integral $\int_\Omega |f(x)|^p\ \textnormal{d}x$ (resp.\ the essential supremum of $|f|$) is finite. The \emph{upper integral} of a function $f\colon \Omega \to \mathbb{R}$ is
\begin{equation}
    \upint_\Omega f(x)\ \textnormal{d}x:=\inf_{\substack{h\in L^1(\Omega)\\ f\leq  h}}\ \int_\Omega h(x)\ \textnormal{d}x.
\end{equation}

We say that $f\colon \Omega \to X$ is  \emph{essentially separably valued} if there is a Lebesgue nullset $N\subset \Omega$ such that $f(\Omega \setminus N)$ is separable. If $f\colon \Omega \to X$ is measurable and essentially separably valued then $f$ is approximately continuous at almost every $x\in \Omega$, see Theorem~2.9.13 in \cite{Fed:65}. The latter means that for every $\varepsilon >0$ one has
\[
\lim_{\delta\downarrow 0}\frac{\mathcal{L}^n\left(\left\{y \in B(x,\delta)\ | \  d(f(y),f(x))\geq \varepsilon \right\}\right)}{\mathcal{L}^n\left(B(x,\delta)\right)}=0.
\] Indeed it is consistent with ZFC to assume that every measurable $f\colon \Omega \to X$ is essentially separably valued, see Sections~2.1.6 and~2.3.6 in~\cite{Fed:65}.

We denote by $L^p(\Omega; X)$ the space of measurable essentially separably valued functions $f\colon \Omega \to X$ such that $x\mapsto d(f(x), z_0)$ 
defines an element of $L^p(\Omega)$ for some (or equivalently every) $z_0\in X$. To functions $f\in L^1(\Omega; V)$ one can associate a vector in $V$ which is called the \emph{Bochner integral} and denoted by $\int_\Omega f(x) \ \textnormal{d}x$. The idea of the construction is to set 
\begin{equation}
    \int_\Omega f(x) \ \textnormal{d}x=\sum_{i=1}^n \mathcal{L}^n(A_i)\cdot v_i
\end{equation}
for simple functions of the form $f(x)=\sum_{i=1}^n\chi_{A^i}(x)\cdot v_i$, and then prove that functions in $L^1(\Omega; V)$ can be suitable approximated by simple functions such that the Bochner integrals of these converge. See \cite[Section~3.2]{HKST2015} for the details. 

For $f\in L^1(\Omega ;V)$ and $E\subset \Omega$ with $\mathcal{L}^n(E)>0$ we define 
\begin{equation}
    \dashint_E f(x) \ \textnormal{d}x:=\frac{1}{\mathcal{L}^n(E)}\cdot \int_E f(x)\ \textnormal{d}x=\frac{1}{\mathcal{L}^n(E)}\cdot \int_{\Omega} \chi_E(x)\cdot f(x)\ \textnormal{d}x.
\end{equation}
We will need the following observation.
\begin{lem}
\label{lem:leb-dens}
    Let $f\in L^\infty(\Omega;V)$ and $0<C\leq 1$. Further let $x\in \Omega$ be a point of approximate continuity of $f$ and $(E^\delta)_{\delta>0}$ be a family of measurable sets such that  $E^\delta\subset B(x,\delta)$ and 
    $\mathcal{L}^n\left(E^\delta\right)\geq C \cdot \mathcal{L}^n(B(x,\delta))$.  Then \[
    \dashint_{E^\delta}f(y) \ \textnormal{d}y \to f(x) \quad \textnormal{ as }\delta \to 0.
    \]
\end{lem}
 The straightforward proof of Lemma~\ref{lem:leb-dens} is left to the reader. Since every function $f\in L^\infty(\Omega;V)$ is approximately continuous almost everywhere, Lemma~\ref{lem:leb-dens} provides a variant of the Lebesgue density theorem for Banach valued functions and the Bochner integral.

We denote by $V^*$ the dual Banach space of $V$ and by $\wlim$ the convergence of dual vectors with respect to the weak* topology on $V^*$. For  $v\in V$, $\Lambda\in V^*$ and $f\colon \Omega \to V^*$ we define $\langle v , \Lambda \rangle :=\Lambda(v)$ and the function $f_v\colon \Omega \to \mathbb{R}$ by \[f_v(x):=\langle v, f(x)\rangle.\]
A map $f\colon \Omega \to V^*$ is called \emph{weak* measurable} if $f_v$ is measurable for every $v\in V$. For $p\in [1,\infty)$ (resp.\ $p=\infty$) the space $L^p_*(\Omega; V^*)$ is defined as the collection of those weak* measurable functions $f\colon \Omega \to V^*$ such that the upper integral $\upint_\Omega \|f(x)\|^p\ \textnormal{d}x$ (resp.\ the essential supremum of $\|f\|$) is finite. The \emph{Gelfand integral} of a function $f\in L^1_*(\Omega; V^*)$ is the dual vector $\int_\Omega f(x)\ \textnormal{d}x\in V^*$ which is uniquely defined by the equality
\begin{equation}
    \left\langle v , \int_\Omega f(x) \ \textnormal{d}x\right\rangle =\int_\Omega f_v(x)\ \textnormal{d}x
\end{equation}
for every $v\in V$. Note that Bochner integral and Gelfand integral coincide whenever $f\in L^1(\Omega; V^*)$, and hence our choice of notation does not create ambiguity. For more details the reader is referred to \cite[Section~2.2]{CE:2021} and the references therein.
\subsection{Definitions of metric valued Sobolev maps}
%Many definitions of the Sobolev space $W^{1,p}(\Omega ;X)$ have been proposed over the last 30 years.\todo{There is almost the same sentence in the begining of Subsection 1.1.} 
The abundance of different approaches concerning the definition of $W^{1,p}(\Omega;X)$ stems from the search for the 'right' generalization when the domain~$\Omega$ is a singular metric measure space. However if $\Omega \subset \mathbb{R}^n$ is a bounded domain and $p\in (1,\infty)$ all of these are equivalent, see \cite[Chapter 10]{HKST2015}, \cite{Reshetnyak2004} and \cite{AGS:13}. In this article we will stick to the following definition which goes back to Reshetnyak \cite{Reshetnyak97, Reshetnyak2004}.% Like the definition that we propose it cannot generalize to metric measure spaces $\Omega$ since it involves smooth test functions on $\Omega$.
\begin{defn}
\label{definition:WX-space}
The space $W^{1,p}(\Omega; X)$ consists of those $f\in L^p(\Omega; X)$ such that
\begin{enumerate}[(A)]
\item for every Lipschitz function $\varphi:X\to\mathbb R$  one has $\varphi \circ f \in W^{1,p}(\Omega)$, and
\item there is a function $g \in L^p(\Omega)$ such that for every Lipschitz $\varphi\colon X\to \mathbb{R}$ one has
\[ |\nabla (\varphi\circ f)(x)| \leq \operatorname{Lip}(\varphi)\cdot g(x) \quad \textnormal{for a.e.\ }x\in \Omega.\]
\end{enumerate}
A function $g$ as in (B) is called a Reshetnyak upper gradient of $f$.
\end{defn}
Notice that if $Y$ is another complete metric space and $\iota \colon X\to Y$ an isometric embedding then 
\begin{equation}
\label{eq:amb-indep}
    W^{1,p}(\Omega ;X)=\left\{f\colon \Omega \to X \ | \ \iota \circ f \in W^{1,p}(\Omega ;Y)\right\}.
\end{equation}
This is implied by the McShane extension lemma, see e.g.\ \cite[Section 1.19]{Hei:03}. In particular once we have proven Theorem~\ref{thm:main} it follows that the definition of $W^{1,p}_*(\Omega ; X)$ given in \eqref{eq:def-w*-met} does not depend on the chosen isometric embedding $\iota$ of $X$ into a dual Banach space. For example the following universal constructions are always possible.
\begin{ex}\label{ex:emb}
    \begin{enumerate}[(i)]
\item \label{it:emb-kur} We denote $\ell^p(X):=L^p(X,\mathcal{P}(X), \#)$ where $\#$ is the counting measure on $X$. Then \cite[Theorem~243G]{Fre03} implies that $\ell^\infty(X)$ is the dual space of $\ell^1(X)$. One can now choose a basepoint $x_0\in X$ and define the Kuratowski embedding $\kappa^{X,x_0} \colon X\to \ell^\infty(X)$ by setting 
\begin{equation}
    \left(\kappa^{X,x_0}(x)\right)(y):=d(x,y)-d(y,x_0).
\end{equation}
It is straightforward to check that this well-defined and an isometric embedding, see e.g.\ \cite[p.\ 11]{Hei:03}. If $X$ is bounded one can avoid the choice of a basepoint and define $\kappa^X\colon X \to \ell^{\infty}(X)$ by setting $(\kappa^X(x))(y):=d(x,y)$. 
\item \label{it:emb-free} The following other universal construction is closely related to Lipschitz free spaces as studied e.g.\ in \cite{Wea99, GK03, DKO20}. We choose a basepoint $x_0\in X$ and denote by $\textnormal{Lip}_{x_0}(X)$ the space of Lipschitz functions $f\colon X\to\mathbb{R}$ such that $f(x_0)=0$. $\textnormal{Lip}_{x_0}(X)$ is endowed with the norm which associates to each $f$ its optimal Lipschitz constant. We denote $\mathcal{F}(X,x_0):=\left(\textnormal{Lip}_{x_0}(X)\right)^*$. Then an isometric embedding $\Psi^{X,x_0}\colon X \to \mathcal{F}(X,x_0)$ can be defined by setting
\begin{equation}
    \left\langle f, \Psi^{X,x_0}(x)\right\rangle:=f(x)
\end{equation}
for $f\in \textnormal{Lip}_{x_0}(X)$ and $x\in X$.
\end{enumerate}
\end{ex}
Actually Definition~\ref{definition:WX-space} is a slight variation on the original definitions of the Sobolev space $W^{1,p}(\Omega;X)$ that Reshetnyak proposed in \cite{Reshetnyak97} and \cite{Reshetnyak2004}. Although all three are equivalent when $X$ is separable, there are in general slight differences which are discussed by the subsequent remark.
\begin{rem}
\label{rem:resh-def}
\begin{enumerate}[(i)]
\item \label{it:rem-resh-def-new}
In \cite{Reshetnyak2004} the only difference compared to Definition~\ref{definition:WX-space} is that Reshetnyak does not assume a priori that $f\in L^p(\Omega;X)$. However the measurability of $f$ is implied from (A) since the distance function to any given closed set is $1$-Lipschitz. Thus, as discussed in Section~\ref{subsec:meas-int}, it is consistent with ZFC to assume that the two definitions are equivalent.
\item In his previous article \cite{Reshetnyak97} Reshetnyak restricts to the postcomposition with distance functions of the form $\varphi=d(x,-)$. For the general case of nonseparable target spaces this does not seem to be a suitable definition. E.g.\ with such definition any injective map into a discrete metric space $X$ would be of Sobolev regularity and an ambient space independence as observed in \eqref{eq:amb-indep} would no longer hold true.
\end{enumerate}
\end{rem}
\subsection{Metric derivatives}
The \emph{length} of a continuous curve $\gamma \colon (a,b)\to X$ is defined as 
\begin{equation}
    \ell(\gamma):=\sup_{a<t_1<\dots<t_k<b}\ \sum_{i=1}^{k-1} d(\gamma(t_i),\gamma(t_{i+1})).
\end{equation}
The curve $\gamma$ is called \emph{rectifiable} if $\ell(\gamma)$ is finite. 
It is called \emph{absolutely continuous} if it is rectifiable and 
$t\mapsto \ell(\gamma|_{(a,t)})$ defines an absolutely continuous 
function $(a,b)\to \mathbb{R}$. We say that $\gamma$ is \emph{metrically differentiable} at $t_0\in (a,b)$ if 
\begin{equation}
    \textnormal{md}\gamma(t_0):= \lim_{t\to t_0}\frac{d(\gamma(t),\gamma(t_0))}{|t-t_0|}\in [0,\infty)
\end{equation}
exists. For a proof of the following lemma see e.g.\ \cite[Proposition 4.4.25]{HKST2015}.
\begin{lem}
    Let $\gamma \colon (a,b)\to X$ be absolutely continuous. Then $\gamma$ is metrically differentiable at almost every $t\in(a,b)$, the almost everywhere defined function $t\mapsto \textnormal{md}\gamma(t)$ is measurable and 
    \begin{equation}
        \ell(\gamma)=\int_a^b \textnormal{md} \gamma (t)\ \textnormal{d}t.
    \end{equation}
\end{lem}
We say that $f\colon \Omega \to \mathbb{R}^n$ has a \emph{$j$-th metric partial derivative} at $x\in \Omega$ if 
\begin{equation}
    \textnormal{m}\partial_j f(x):=\lim_{h\to 0} \frac{d(f(x+h\cdot e_j),f(x))}{|h|}\in [0,\infty)
\end{equation}
exists. 
A map $f\colon \Omega \to X$ is \emph{absolutely continuous on almost every line segment parallel to the coordinate axes} if for every $j=1,\dots, n$ there is a family of lines $L$, which are parallel to the $x_j$-axis, such that $\mathcal{L}^n(\Omega\setminus\bigcup L)=0$ and for every $l\in L$ the map $f$ is absolutely continuous when restricted to a line segment contained in $l\cap \Omega$. In this case we write $f\in \textnormal{ACL}(\Omega ;V)$. It turns out that if $f\in \textnormal{ACL}(\Omega ;V)$ is measurable then $\textnormal{m}\partial_j f(x)$ exists at almost every $x\in \Omega$ and the almost everywhere defined function $x\mapsto \textnormal{m}\partial_j f(x)$ is measurable, see \cite[Lemma 2.4]{CJP22}.

We call $f\colon \Omega \to X$ a \emph{representative} of $g\colon \Omega \to X$ and write $f\equiv g$ if $f(x)=g(x)$ for almost every $x\in \Omega$. In particular $f\equiv g$ implies that $f$ is measurable if and only if $g$ is measurable. It turns out that $\gamma\in W^{1,1}((a,b),X)$ if and only $\gamma$ has a representative that is absolutely continuous. In higher dimensions we have the following implication (compare \cite[Theorem 3.1]{CJP22} and \cite[Lemma 2.13]{HT:08}).
\begin{prop}
\label{lem:lemmaAC}
Let $f\in W^{1,p}(\Omega; X)$. Then $f$ has a representative $\widetilde{f}\in \textnormal{ACL}(\Omega;X)$. 
Moreover, for every Reshetnyak upper gradient $g$ of $f$ one has
\begin{equation}\label{eq:lemmaACest}
\textnormal{m}\partial_j \widetilde{f}(x)\leq g(x)
\end{equation}
for almost every $x\in \Omega$.
\end{prop}

\section{Weak weak* derivatives}
\label{sec:weak-weak*}
Throughout this section let $\Omega \subset \mathbb{R}^n$ be a bounded domain, $p\in [1,\infty]$ and $V^*$ be a dual Banach space. 
\subsection{The inclusion $W^{1,p}(\Omega ;V^*)\subset W^{1,p}_*(\Omega ;V^*)$}\label{sec:hard_inc} 
This part relies on the notion of ultralimits. 
An \emph{ultrafilter} on $\mathbb{N}$ is a finitely additive measure 
$\omega \colon \mathcal{P}(\mathbb{N})\to \{0,1\}$ such that $\omega(\mathbb{N})=1$.
An ultrafilter is called \emph{nonprincipal} if $\omega(S)=0$ for every finite set $S$. The axiom of choice implies that a nonprincipal ultrafilter $\omega$ on $\mathbb{N}$ exists (see \cite[Corollary 2.6.2]{Gol98}) and we fix such for the remainder of this section. For a sequence $(a_n)_{n\in \mathbb{N}}$ of real numbers we say that $a\in \mathbb{R}$ is the \emph{$\omega$-limit} of $(a_n)$ and write $\lim_{n\in \omega} a_n=a$ if
\begin{equation}
    \omega\left(\left\{m\in \mathbb{N} \ \colon \ a_m \in (a-\epsilon, a+ \epsilon)\right\}\right)=1
\end{equation}
for every $\epsilon >0$. It turns out that $\lim_{n\in \Omega}a_n$ exists for every bounded sequence $(a_n)\subset \mathbb{R}$ (compare e.g.\ \cite[Proposition 4.4]{AKP19}) and that \[
\lim_{n\in \Omega}a_n=\lim_{n\to \infty} a_n
\]
if the latter exists. Furthermore, if $(a_n)$ and $(b_n)$ are bounded sequences of real numbers it is easy to check that
\[
\lim_{n\in \Omega}(a_n+b_n)=\lim_{n\in \Omega} a_n+\lim_{n\in \Omega} b_n \quad \textnormal{and} \quad \lim_{n\in \Omega}\left(a_n\cdot b_n\right)=\left(\lim_{n\in \Omega} a_n\right)\cdot\left(\lim_{n\in \Omega} b_n\right).
\]
\begin{thm}
\label{theorem:main}
One has 
\begin{equation}
    W^{1,p}(\Omega ; V^*)\subset W_*^{1,p}(\Omega ; V^*).
\end{equation}
Furthermore, if $g\in L^p(\Omega)$ is an Reshetnyak upper gradient of $f\in W^{1,p}(\Omega ; V^*)$ then $f$ has a $j$-th partial weak weak* derivative $f^j\in L^p_*(\Omega ;V^*)$ such that $\|f^j(x)\|\leq g(x)$ for almost every $x\in \Omega$.
\end{thm}
\begin{proof}
    Let $\widetilde{f}$ be a representative of $f$ as in Lemma~\ref{lem:lemmaAC}. We fix a sequence $(\delta_n)$ of positive reals such that $\delta_n\to 0$ as $n\to \infty$. For $x \in \Omega$ we define $f^j(x)\colon V\to \mathbb{R}$ by setting
    \begin{equation}
        \langle v, f^j(x) \rangle := \begin{cases}
\lim_{n\in \Omega} \frac{\widetilde{f}_v(x+\delta_n\cdot e_j)-\widetilde{f}_v(x)}{\delta_n} &, \textnormal{ if }\textnormal{m}\partial_j\widetilde{f}(x) \textnormal{ exists}\\
0 &, \textnormal{ otherwise.}
        \end{cases}
    \end{equation} 
    Note that 
    \begin{equation}
    \label{eq:pre-fact}
        \left|\widetilde{f}_v(x+\delta_n\cdot e_j)-\widetilde{f}_v(x)\right|\leq \|v\|\cdot\|\widetilde{f}(x+\delta_n\cdot e_j)-\widetilde{f}(x)\|.
    \end{equation}
    By \eqref{eq:pre-fact} the function $f^j(x)\colon V\to \mathbb{R}$ is well-defined and if $\textnormal{m}\partial_j\widetilde{f}(x)$ exists then
    \begin{equation}
    \label{eq:norm-est}
        \langle v , f^j(x)\rangle \leq \|v\|\cdot \textnormal{m}\partial_j\widetilde{f}(x).
    \end{equation}
    Furthermore, by the linearity of ultralimits, $f^j(x)\colon V\to \mathbb{R}$ is linear and hence $f^j(x)\in V^*$ for every $x\in \Omega$.

    Now fix $v\in V$. Then \begin{equation}
    \label{eq:pointwise-main}
    f^j_v(x)=\partial_j \widetilde{f}_v(x)
    \end{equation}
    for every $x\in \Omega$ such that $\textnormal{m}\partial_j\widetilde{f}(x)$ and $\partial_j \widetilde{f}_v(x)$ exist. In particular we have that \eqref{eq:pointwise-main} holds for almost every $x\in \Omega$. Hence $f^j_v$ is measurable and a weak derivative of $\widetilde{f}_v\in W^{1,p}(\Omega)$. 
    
    Since $v\in V$ was arbitrary, $f^j$ is weak* measurable. Furthermore by  \eqref{eq:norm-est} and Proposition~\ref{lem:lemmaAC} one has 
    \begin{equation}
        \|f^j(x)\|\leq g(x)
    \end{equation}
    for almost every $x\in \Omega$ and hence in particular $f^j\in L^p_*(\Omega ;V^*)$. Finally, that $f^j_v$ is a $j$-th weak derivative of $\widetilde{f}_v$ for every $v\in V$, implies that $f^j$ is a $j$-th weak weak* derivative of $\widetilde{f}$ and hence also of $f$.
\end{proof}
\subsection{The inclusion $W^{1,p}_*(\Omega ;V^*)\subset W^{1,p}(\Omega ;V^*)$}
\label{sec:triv_inc}
This inclusion already follows from the proof of Theorem~1.4 in~\cite{CE:2021} However, in \cite{CE:2021} we worked with a slightly more restrictive definition of $L^p_*(\Omega ;V^*)$ and hence $W^{1,p}_*(\Omega ; V^*)$. Hence we shortly recall the critical step for the readers convenience. To this end we need the following result which is a direct consequence of \cite[Proposition~3.4]{CE:2021} and \cite[Proposition~2.16]{HT:08}
\begin{prop}
\label{Prop:W*=R*}
    Let $f\in L^p(\Omega; V^*)$. Then $f\in W^{1,p}(\Omega; V^*)$ if and only if 
    \begin{enumerate}[(A*)]
\item for every $v\in V$ the function $f_v$ lies in  $W^{1,p}(\Omega)$, and
\item \label{item:WX2} there is a function $g \in L^p(\Omega)$ such that for every $v\in V$ one has
\[ |\nabla (f_v)(x)| \leq \|v\| \cdot g(x) \quad \textnormal{for a.e.\ }x\in \Omega.\]
\end{enumerate}
\end{prop}
To conclude that $W^{1,p}_*(\Omega ;V^*)\subset W^{1,p}(\Omega ;V^*)$ we proceed as in the first part of the proof of Proposition~3.5 in \cite{CE:2021}: Let $f\in W^{1,p}_*(\Omega ; V^*)$ and for $j\in \{1,\dots ,n\}$ choose a $j$-th weak weak* derivative $f^j\in L^p_*(\Omega; V)$ of $f$. By definition of the upper integral, there are functions $g^j\in L^p(\Omega)$ such that $\|f^j(x)\|\leq g^j(x)$ for every $x\in \Omega$. 
Since $f\in L^p(\Omega ; V^*)$ the function $f_v$ lies in $L^p(\Omega)$ for every $v \in V$ . Furthermore for $j\in \{1, \dots ,n \}$ and $v\in V$ the function $f^j_v$ is a $j$-th weak partial derivative of $f_v$. Hence $f_v \in W^{1,p}(\Omega)$ and 
\[
|\nabla f_v (x)|= \left( \sum_{j=1}^n  \left|f^j_v(x)\right|^2\right)^{1/2}
\leq \left( \sum_{j=1}^n |g^j(x)|^2\cdot ||v||^2\right)^{1/2}\leq \|v\| \cdot \max_{j\in\{1, \dots n\}}|g^j(x)|
\]
for a.e.\ $x\in \Omega$. Setting $g:=\max_{j\in \{1,\dots,n\}}\ |g^j| \in L^p(\Omega)$, we conclude from Proposition~\ref{Prop:W*=R*} that $f\in W^{1,p}(\Omega;V^*)$.
\subsection{Uniqueness of weak weak* derivatives and Sobolev norms}\label{sec:uniq} We say that $g\colon \Omega \to V^*$ is a \emph{weak* representative} of $f\colon \Omega \to V^*$ and write $f \equiv_* g$ if for every $v\in V$ one has that $f_v\equiv g_v$. Certainly $f\equiv g$ implies that $f \equiv_* g$. Concerning the validity of the inverse implication we have the following examples and counterexamples.
\begin{ex}
\label{ex:weak*-equiv}
\begin{enumerate}[(i)]
    \item If $V$ is separable then $f\equiv_* g$ implies that $f \equiv g$ for functions $f,g\colon \Omega \to V^*$. Indeed if $f \equiv_* g$ and $S\subset V$ is a countable dense subset then for each $v\in S$ there is a nullset $N_v \subset \Omega$ such that $f_{v}=g_{v}$ on $\Omega \setminus N_v$. In particular, for the nullset $N:=\bigcup_{v\in S}N_v$ it follows that 
\begin{equation}
\label{eq:ae-eq-f-g}
\langle v, f(x)\rangle=\langle v, g(x)\rangle
\end{equation}
for every $v \in S$ and $x\in \Omega \setminus N$. Since $S$ is dense, we conclude from~\eqref{eq:ae-eq-f-g} that $f=g$ on $\Omega \setminus N$ and hence that $f\equiv g$.
    \item \label{it:w*-counterex}
    For a given function $\Phi \colon \Omega \to \mathbb{R}$ we can define $h^\Phi\colon \Omega \to \ell^\infty(\Omega)$ by setting \[
h^\Phi(x):=\Phi(x)\cdot \chi_{\{x\}}.\]
Then $h_\Phi\equiv_* 0$ since any $g\in \ell^1(\Omega)$ is supported on a countable set $S$ and \[
\langle g, h^\Phi(x)\rangle=0
\]
for every $x\in \Omega\setminus S$. On the other hand $h_\Phi\equiv 0$ only if $\Phi\equiv 0$.
\end{enumerate}
    
\end{ex}

It turns out that weak weak* derivatives are not unique up to the choice of a representative but instead only up to the choice of a weak* representative.
\begin{lem}
\label{lem:w*-equiv}
    Let $f\in L^1(\Omega; V^*)$. Further let $f^j\colon \Omega \to V^*$ be a $j$-th weak weak* derivative of $f$ and $h\colon \Omega \to V^*$. Then $h\equiv_* f^j$ if and only if $h$ is a $j$-th weak weak* derivative of $f$.
\end{lem}
\begin{proof}
    Assume $h\equiv_* {f}^j$. Then for every $v\in V$ one has  $h_v\equiv f^j_v$. In particular $h_v \colon \Omega \to \mathbb{R}$ is measurable and 
    \begin{equation}
        \int_\Omega \varphi(x)h_v(x)\ \textnormal{d}x=\int_\Omega \varphi(x)f^j_v(x)\ \textnormal{d}x
    \end{equation}
    for every $\varphi \in C^\infty_0(\Omega)$. Thus we conclude that $h$ is weak* measurable and a $j$-th weak weak* derivative of $f$.

    For the other implication assume $h$ is a $j$-th weak weak* derivative of $f$. Then for every $v\in V$ both $f^j_v$ and $h_v$ are weak derivatives of $f_v$. Thus the uniqueness of weak derivatives for real valued functions implies that $h_v\equiv f^j_v$. Since this holds for every $v\in V$ we conclude that $h \equiv_* f^j$.
    \end{proof}

From Example~\ref{ex:weak*-equiv}\eqref{it:w*-counterex} and Lemma~\ref{lem:w*-equiv} we see that even constant maps can have weak weak* derivatives $h$ for which the function $x \mapsto \|h(x)\|$ is nonmeasurable or $\upint_\Omega \|h(x)\|^p \ \textnormal{d}x$ is arbitrary large (or even infinite). Nevertheless one can define a (semi)norm on $W^{1,p}_*(\Omega ;V^*)$ by setting
\begin{equation}
    \|f\|_{W^{1,p}_*(\Omega; V^*)}:=\left(\int_\Omega \|f(x)\|^{p}\ \textnormal{d}x\right)^{\frac{1}{p}}+\inf \left( \upint_\Omega\left|\left(f^1(x),f^2(x),\dots,f^n(x)\right)\right|^p \textnormal{d}x\right)^{\frac{1}{p}}
\end{equation}
where $|.|$ denotes the product norm on $(V^*)^n$ and the infimum ranges respectively over all partial weak weak* derivatives $f^j$ of $f$. The estimates in Sections~\ref{sec:hard_inc} and~\ref{sec:triv_inc} show that
\begin{equation}
    \frac{1}{\sqrt{n}}\cdot\|f\|_{W^{1,p}}\leq \|f\|_{W^{1,p}_*}\leq \sqrt{n} \cdot \|f\|_{W^{1,p}}
\end{equation}
for every $f\in W^{1,p}(\Omega ;V^*)$. Here as usual the Sobolev norm on $W^{1,p}(\Omega ;V^*)$ is defined by
\begin{equation}
    \|f\|_{W^{1,p}}:=\left(\int_\Omega \|f(x)\|^{p}\ \textnormal{d}x\right)^{\frac{1}{p}}+\inf \left(\int_\Omega\left|g(x)\right|^p \textnormal{d}x\right)^{\frac{1}{p}}
\end{equation}
where the infimum ranges over all Reshetnyak upper gradients of $f$. In particular the Sobolev norms  $\|\cdot\|_{W^{1,p}(\Omega ;V^*)}$ and $\|\cdot\|_{W^{1,p}_*(\Omega ;V^*)}$ are equivalent.

\subsection{Measurability of the norm of weak weak* derivatives}
\label{sec:meas-norm}
We remark again that in contrast to \cite{CE:2021} we do not assume the measurability of $x\mapsto \|f(x)\|$ for functions $f\in L^1_*(\Omega;V^*)$. In particular we cannot prove the existence of weak weak* derivatives with this property. The following Lemma however shows that it is indeed possible in the one dimensional case. 
\begin{lem}
\label{lem:norm-meas}
Let $\gamma\colon (a,b)\to V^*$ be absolutely continuous. Then %for every weak weak* derivative $\eta \in L^1_*$ of $\gamma$ one has
%\begin{equation}
%    \|\eta(t)\|\geq m\partial \gamma(t)
%\end{equation}
%for almost every $t\in [a,b]$. In particular 
there is a weak weak* derivative $\eta\in L^1_*((a,b);V^*)$ of $\gamma$ such that $t\mapsto \|\eta(t)\|$ is measurable and \begin{equation}
\label{eq:eta_md}
    \|\eta(t)\|=\textnormal{md}\gamma(t)
\end{equation} for almost every $t\in (a,b)$.
\end{lem}
\begin{proof}
    By the proof of Theorem~\ref{theorem:main} there is a weak weak* derivative $\eta$ of $\gamma$ such that
    \begin{equation}
        \|\eta(t)\|\leq \textnormal{md}\gamma (t)
    \end{equation}
    for every $t\in (a,b)$ at which $\gamma$ is metrically differentiable. For each $k\in \mathbb{N}$ we choose $a< t^1_k<t^2_k\dots <t^{n_k}_k< b$ such that
    \begin{equation}
        \sum^{n_k-1}_{l=1}\|\gamma(t^l_k)-\gamma(t^{l+1}_k)\|\geq \ell(\gamma)-\frac{1}{k}
    \end{equation}
    and respectively $v^{kl}\in V$ with $\|v^{kl}\|\leq 1$ such that 
    \begin{equation}
        \langle v^{kl},\gamma(t^{l}_k)-\gamma(t^{l+1}_k)\rangle \geq \|\gamma(t^{l}_k)-\gamma(t^{l+1}_k)\|-\frac{1}{k\cdot n_k}.
    \end{equation}
    We define $\nu \colon (a,b)\to \mathbb{R}$ by
    \begin{equation}
        \nu(t):=\sup_{k\in \mathbb{N}} \sup_{l\in \{1,\dots, n_k\}} \left|\eta_{v^{kl}}(t)\right|.
    \end{equation}
    Then $\nu$ is measurable and for almost every $t\in (a,b)$ one has
    \begin{equation}
    \label{eq:lattice-ineq}
        \nu(t)\leq \|\eta(t)\|\leq \textnormal{md}\gamma(t).
    \end{equation}
    Furthermore for every $k\in \mathbb{N}$ we have that
    \begin{align}
        \int_a^b\nu(t)\ \textnormal{d}t&\geq \sum_{l=1}^{n_k-1} \int_{t^l_k}^{t^{l+1}_k}|\eta_{v^{kl}}(t)|\ \textnormal{d}t\\
        &\geq \sum_{l=1}^{n_k-1} \langle v^{kl},\gamma(t^{l}_k)-\gamma(t^{l+1}_k)\rangle\\
        &\geq \ell(\gamma)-\frac{2}{k}.
    \end{align}
    and hence
    \begin{equation}
        \int_a^b\nu(t)\ \textnormal{d}t=\ell(\gamma)=\int_a^b \textnormal{md}\gamma(t)\ \textnormal{d}t.
    \end{equation}
    Together with \eqref{eq:lattice-ineq} this implies that $\nu(t)=\textnormal{md}\gamma(t)$ for almost every $t\in (a,b)$ and hence \eqref{eq:eta_md}. Since $t\mapsto \textnormal{md}\gamma(t)$ is measurable, we conclude that so is $t\mapsto \|\eta(t)\|$.
\end{proof}
Unfortunately it seems unclear how to derive from Lemma~\ref{lem:norm-meas} the measurability of $x\mapsto \|f^j(x)\|$ for Sobolev maps $f$ defined on a higher dimensional domain $\Omega$. The problem is that if the $1$-dimensional line sections of a set $S\subset \Omega$ are $\mathcal{H}^1$-nullsets this does not imply the $\mathcal{L}^n$-measurability of $S$, see e.g.\ \cite[Example 10.21]{GO64}. 
\subsection{Pointwise weak* derivatives}
\label{sec:pointwise}
We say that a map $f\colon \Omega \to V^*$ is \emph{weak* differentiable} at $x\in \Omega$ in the $x_j$-direction if there is $\partial_j f(x)\in V^*$ such that 
\[
\frac{1}{h}\cdot \left(f(x+h\cdot e_j)-f(x)\right)  \wlim \partial_j f(x) \quad \textnormal{as }h\to 0.
\]
The proof of the equality of $W^{1,p}_*(\Omega ;V^*)$ and $W^{1,p}(\Omega ;V^*)$ for dual-to-separable Banach spaces in \cite{CJP22, CE:2021} heavily relies on the observation that absolutely continuous curves in dual-to-separable Banach spaces are weak* differentiable almost everywhere. The following examples show that this is not true for arbitrary dual spaces~$V^*$.

\begin{ex}
\label{ex:not-diff}
\begin{enumerate}[(i)]
\item We claim that the Kuratowski embedding of the unit interval $\kappa=\kappa^{(0,1)}\colon (0,1)\to \ell^\infty((0,1))$ as defined in Example~\ref{ex:emb}.\eqref{it:emb-kur} is nowhere weak* differentiable. To this end fix $t_0\in (0,1)$. Then 
\begin{equation}
    \left\langle \chi_{\{t_0\}}, \kappa(t)\right\rangle=|t-t_0|
\end{equation}
is not differentiable in $t$ at $t=t_0$. In particular $\kappa$ cannot be weak* differentiable at $t_0$.

\item Similarly $\Psi=\Psi^{(0,1),0}\colon (0,1)\to \mathcal{F}((0,1),0)$ as defined in Example~\ref{ex:emb}.\eqref{it:emb-kur} is nowhere weak* differentiable. Indeed for every $t_0\in (0,1)$ there is a function $f\in \textnormal{Lip}_0((0,1))$ such that $f$ is not differentiable at $t_0$ (e.g.\ $f(t)=|t-t_0|-t_0$). This however means that \[\Psi_f(t)=f(t)\] is not differentiable in $t$ at $t=t_0$, and hence $\Psi$ not weak* differentiable at $t_0$.
\end{enumerate}
\end{ex}

To obtain a characterization of dual spaces $V^*$ in which absolutely continuous curves are weak* differentiable we propose the following definition.

\begin{defn}
\label{def:leb-dens}
    We say that a function $f\in L^{1}_*((0,1);V^*)$ satisfies the weak* Lebesgue density theorem if for almost every $t\in (0,1)$ one has
    \begin{equation}
        \dashint_t^{t+h}f(s) \ \textnormal{d}s \wlim f(t) \quad \textnormal{as }h \to 0.
    \end{equation}
    We say that $V^*$ has the weak* Lebesgue density property if every $f\in L^{\infty}_*((0,1);V^*)$ has a weak* representative which satisfies the weak* Lebesgue density theorem.
\end{defn}

We have the following two classes of examples of spaces which have the Lebesgue density property.
\begin{ex}
\begin{enumerate}[(i)]
\item If $V$ is separable then $V^*$ has the weak* Lebesgue density property. To prove this we fix $f\in L^\infty_*((0,1) ;V^*)$ and choose a countable dense set $S\subset V$. For $v\in S$ let $N_v\subset (0,1)$ be a nullset such that $f_v$ is approximately continuous on $(0,1)\setminus N_v$. Setting $N:=\bigcup_{v\in S} N_v$ we deduce from Lemma~\ref{lem:leb-dens} 
that
\begin{equation}
        \dashint_t^{t+h}f(s) \ \textnormal{d}s \wlim f(t) \quad \textnormal{as }h \to 0
    \end{equation}
for every $t\in (0,1)\setminus N$. In particular $f$ satisfies the weak* Lebesgue density theorem and hence $V^*$ has the Lebesgue density property.
\item If $V^*$ has the Radon--Nikodým property then $V^*$ has the weak* Lebesgue density property. Indeed every function $f\in L^{\infty}_*((0,1);V^*)$ defines a continuous functional $T_f\colon L^{1}((0,1))\to V^*$ by setting
\begin{equation}
    T_f(h):=\int_0^1 h(s)\cdot f(s)\ \textnormal{d}s.
\end{equation}
Now, if $V^*$ has the Radon--Nikodým property then \cite[Theorem 1.3.16]{HNVW16} implies that there is $g\in L^{\infty}((0,1);V^*)$ such $T_f=T_g$. In particular then $g$ is a weak* representative of $f$ and, since it is measurable, satisfies the Lebesgue density theorem.
\end{enumerate}
\end{ex}
%The following theorem tells us that the weak* Lebesgue density property indeed characterizes the almost everywhere weak* differentiability of absolutely continuous curves.
\begin{thm}
    The following are equivalent:
    \begin{enumerate}[(i)]
        \item \label{it:LDP}
        $V^*$ has the weak* Lebesgue density property.
        \item \label{it:ac}
        Every absolutely continuous curve $\gamma\colon (a,b)\to V^*$ is weak* differentiable almost everywhere.
        %\item \label{it:sob}
        %Every map $f\in W^{1,p}(\Omega; V^*)$ has a representative that is weak* differentiable almost everywhere in the $x_j$-direction.
    \end{enumerate}
\end{thm}
\begin{proof}
    To show that (i) implies (ii) we assume that $V^*$ has the weak* Lebesgue density property. Let $\gamma \colon (a,b)\to V^*$ be a Lipschitz curve. Then $\gamma \in W^{1,\infty}((a,b);V^*)$ and hence, by Theorem~\ref{theorem:main}, $\gamma$ has a weak weak* derivative $\eta \in L^\infty_*((a,b);V^*)$. In particular, for every $v\in V$ one has that $\eta_v=\gamma_v'$ almost everywhere and hence for every $s,t\in (a,b)$ that
    \begin{equation}
        \gamma(t)-\gamma(s)=\int_s^t \eta(r) \ \textnormal{d}r.
    \end{equation}
    Since $V^*$ has the weak* Lebesgue density property we may without loss of generality assume that $\eta$ satisfies the weak* Lebesgue density theorem. Then for almost every $t\in (a,b)$ one has
    \begin{equation}
        \frac{1}{h} \cdot (\gamma(t+h)-\gamma(t)) =\dashint_t^{t+h}\eta(s)\ \textnormal{d}s \wlim \eta(t).
    \end{equation}
    In particular for such $t$ the Lipschitz curve $\gamma$ is weak* differentiable at $t$ with weak* derivative $\eta(t)$.

    Now let $\gamma\colon (a,b)\to V^*$ be a general absolutely continuous curve. Set $L:=b+\ell(\gamma)$ and define $\Phi \colon (a,b)\to (a,L)$ by 
    \[
    \Phi(t):=t+\ell(\gamma|_{(a,t)}).
    \]
    Then $\Phi$ is an absolutely continuous homeomorphism such that $\Phi^{-1}$ is $1$-Lipschitz. Furthermore
    \begin{equation}
       \|\gamma(t)-\gamma(s)\|\leq \ell\left(\gamma|_{[s,t]}\right)\leq |\Phi(t)-\Phi(s)|
    \end{equation}
    for every $s,t\in (a,b)$, and hence $\bar{\gamma}:=\gamma \circ \Phi^{-1}$ is $1$-Lipschitz. Note that the curve $\gamma=\bar{\gamma}\circ \Phi$ is weak* differentiable at $t\in (a,b)$ whenever $\Phi$ is differentiable at $t$ and $\bar{\gamma}$ is weak* differentiable at $\Phi(t)$. Since $\Phi$ is differentiable almost everywhere, $\bar{\gamma}$ is weak* differentiable almost everywhere, and both $\Phi$ and $\Phi^{-1}$ preserve nullsets, we deduce that $\gamma$ is weak* differentiable almost everywhere.

    Next we prove that (ii) implies (i). So assume that absolutely continuous curves in $V^*$ are weak* differentiable almost everywhere and let $\eta \in L^\infty_*((0,1);V^*)$. Then there is $L\in \mathbb{R}$ such that $\|\eta(t)\|\leq L$ for almost every $t\in (0,1)$. We define the curve $\gamma \colon (0,1)\to V^*$ by setting
    \begin{equation}
        \gamma(t):=\int_a^t \eta(s)\ \textnormal{d}s
    \end{equation}
    Then for every $s,t\in (0,1)$ and $v\in V$ one has
    \begin{equation}
        \langle v, \gamma(t)-\gamma(s) \rangle=\int_s^t\eta_v(r)\ \textnormal{d}r\leq L\cdot |t-s|, 
    \end{equation}
    and hence $\gamma$ is $L$-Lipschitz. Thus $\gamma$ is weak* differentiable almost everywhere and a weak weak* derivative $\nu \in L^\infty_*((0,1);V^*)$ can be defined by setting $\nu(x):=\gamma'(x)$ whenever the latter exists. Then $\nu \equiv_* \eta$ and for every $t$ at which $\gamma$ is weak* differentiable
    \begin{equation}
        \dashint_t^{t+h} \nu(s) \ \textnormal{d}s=\dashint_t^{t+h} \eta(s) \ \textnormal{d}s=\frac{1}{h}\cdot (\gamma(t+h)-\gamma(t))\wlim \nu(t)
    \end{equation}
    as  $h\to 0$.
\end{proof}
If $V^*$ is dual-to-separable or has the Radon--Nikodým property then (as was previously observed in \cite{HT:08, CJP22, CE:2021} resp.\ \cite{CJPA2020, Evs21}) it follows that maps $f\in W^{1,p}(\Omega; V^*)$ have representatives that are weak* differentiable almost everywhere. %\todo{I would add here something like "So we obtain all corresponding results from \cite{CE:2021} and \cite{CJP22}". }
It remains an open question whether this is implied more generally by the weak* Lebesgue density property. As in Section~\ref{sec:meas-norm} the problem is that a nonmeasurable property on $\Omega$ may hold almost everywhere along almost every line segment.

\bibliographystyle{plain}
\bibliography{bibliography}
\end{document}